\documentclass[10pt,onecolumn,oneside,a4paper]{article}
\usepackage[yyyymmdd,hhmmss]{datetime}
\usepackage{amsmath,amssymb,amsfonts,amsthm}

\usepackage[pdftex,
            pdfauthor={Claude Gravel and Daniel Panario},
            pdftitle={Evaluation of the constant for the normalized variance of the size of the smallest cycle of a random permutation},
            pdfsubject={Buchstab function and analytic combinatorics},
            pdfkeywords={numerical integration, trapezoidal rule, recurrence relation, complex analysis, combinatorics, Buchstab function, difference-differential equation},
            pdfproducer={latex with hyperref},
            pdfcreator={pdflatex}]{hyperref}

\usepackage{url}
\PassOptionsToPackage{hyphens}{url}
\PassOptionsToPackage{hyphens}{hyperref}
\makeatletter
\g@addto@macro{\UrlBreaks}{\UrlOrds}
\makeatother

\usepackage[dvipsnames]{xcolor}
\allowdisplaybreaks

\usepackage{longtable}

\usepackage{tocbibind}
\usepackage{algorithm,algorithmicx,algpseudocode}%,float}

\makeatletter
\newenvironment{breakablealgorithm}
  {% \begin{breakablealgorithm}
   \begin{center}
     \refstepcounter{algorithm}% New algorithm
     \hrule height.8pt depth0pt \kern2pt% \@fs@pre for \@fs@ruled
     \renewcommand{\caption}[2][\relax]{% Make a new \caption
       {\raggedright\textbf{\ALG@name~\thealgorithm} ##2\par}%
       \ifx\relax##1\relax % #1 is \relax
         \addcontentsline{loa}{algorithm}{\protect\numberline{\thealgorithm}##2}%
       \else % #1 is not \relax
         \addcontentsline{loa}{algorithm}{\protect\numberline{\thealgorithm}##1}%
       \fi
       \kern2pt\hrule\kern2pt
     }
  }{% \end{breakablealgorithm}
     \kern2pt\hrule\relax% \@fs@post for \@fs@ruled
   \end{center}
  }
\makeatother

\floatname{algorithm}{Procedure}

\renewcommand{\algorithmicreturn}[1]{\bgroup\\  ~#1\egroup}
\renewcommand{\algorithmiccomment}[1]{\bgroup\hfill//~#1\egroup}

\theoremstyle{plain} 
\theoremstyle{plain} 
\theoremstyle{remark} \newtheorem{remark}{\textbf{Remark}}
\theoremstyle{plain} \newtheorem{theorem}{\textbf{Theorem}}
\theoremstyle{plain} 
\theoremstyle{plain} 
\theoremstyle{definition}

\makeatletter
\newcommand{\pushright}[1]{\ifmeasuring@#1\else\omit\hfill$\displaystyle#1$\fi\ignorespaces}
\newcommand{\pushleft}[1]{\ifmeasuring@#1\else\omit$\displaystyle#1$\hfill\fi\ignorespaces}
\makeatother

\makeatletter
\let\@@pmod\pmod
\DeclareRobustCommand{\pmod}{\@ifstar\@pmods\@@pmod}
\def\@pmods#1{\mkern4mu({\operator@font mod}\mkern 6mu#1)}
\makeatother

\def\isdef{\buildrel {\rm def} \over =}

\newcounter{NbTabulare} 
\newcommand{\NewTabulare}[1]{\noindent%
        \refstepcounter{NbTabulare}\arabic{NbTabulare}  \textnormal{#1}%
        }

\title{Evaluating the generalized Buchstab function and revisiting the variance of the distribution of the smallest components of combinatorial objects}

\author{%
Claude Gravel\textsuperscript{1}~and~Daniel Panario\textsuperscript{2}\\\phantom{}\\[-4mm]
\textsuperscript{1}\small{Eaglys Inc.}\\[-1.5mm]\small{Tokyo, Japan}\\[-1.5mm]\small{\texttt{claudegravel1980@gmail.com}}\\[1mm]
\textsuperscript{2}\small{School of Mathematics and Statistics}\\[-1.5mm]\small{Carleton University, Canada}\\[-1.5mm]\small{\texttt{daniel@math.carleton.ca}}
}

\date{}%\today~~---~~\currenttime}

\begin{document}

\maketitle
\thispagestyle{empty}

%\tableofcontents

%%%%%%%%%%%%%%%%%%%%%%%%%%%%%%%%%%%%%%%%%%%%%%%%%%%%%%%%%%%%%%%%%%%%%%%%%%%%%%%%
\begin{abstract}
Let $n\geq 1$ and $X_{n}$ be the random variable representing the size of the smallest component of a random combinatorial object made of $n$ elements. A combinatorial object could be a permutation, a monic polynomial over a finite field, a surjective map, a graph, and so on. By a random combinatorial object, we mean a combinatorial object that is chosen uniformly at random among all possible combinatorial objects of size $n$. It is understood that a component of a permutation is a cycle, an irreducible factor for a monic polynomial, a connected component for a graph, etc. Combinatorial objects are categorized into parametric classes. In this article, we focus on the exp-log class with parameter $K=1$ (permutations, derangements, polynomials over finite field, etc.) and $K=1/2$ (surjective maps, $2$-regular graphs, etc.) The generalized Buchstab function $\Omega_{K}$ plays an important role in evaluating probabilistic and statistical quantities. For $K=1$, Theorem $5$ from \cite{PanRic_2001_small_explog} stipulates that $\mathrm{Var}(X_{n})=C(n+O(n^{-\epsilon}))$ for some $\epsilon>0$ and sufficiently large $n$. We revisit the evaluation of $C=1.3070\ldots$ using different methods: analytic estimation using tools from complex analysis, numerical integration using Taylor expansions, and computation of the exact distributions for $n\leq 4000$ using the recursive nature of the counting problem. In general for any $K$, Theorem $1.1$ from \cite{BenMasPanRic_2003} connects the quantity $1/\Omega_{K}(x)$ for $x\geq 1$ with the asymptotic proportion of $n$-objects with large smallest components. We show how the coefficients of the Taylor expansion of $\Omega_{K}(x)$ for $\lfloor x\rfloor \leq x < \lfloor x\rfloor+1$ depends on those for $\lfloor x\rfloor-1 \leq x-1 < \lfloor x\rfloor$. We use this family of coefficients to evaluate $\Omega_{K}(x)$.

\end{abstract}

\section{Introduction}\label{sect_intro}

Let $n\geq 1$ and $X_{n}$ be the random variable representing the size of the smallest component of a random combinatorial object made of $n$ elements. By a random combinatorial object, we mean a combinatorial object that is chosen uniformly at random among all possible combinatorial objects of size $n$. The cardinality of the support of $X_n$ is in principle $n+1$. Since the length of the smallest component cannot be obviously between $\lfloor n/2 \rfloor+1$ and $n-1$ inclusively, the range of $X_{n}$ is therefore $1,2,\ldots, \lfloor n/2 \rfloor$ together with $n$. For some reasons that will become clear hereafter, we add zero probabilities to extend the range of $X_n$ over all integers between $1$ and $n$ inclusively.

Many results pertaining to combinatorial objects and the analytical methods required to understand many of the references in this paper can be found in \cite{FlaSed_BOOK}. Results of Section \ref{sect_approaches} are valid for the class of $n$-objects that contains, permutations, derangements, monic polynomials over a finite fields, just to name a few. Result of Section \ref{sect_gen_bt} applies to all combinatorial objects in the exp-log class. We let readers to consult \cite{FlaSed_BOOK} for the proper definitions of the exp-log class of combinatorial objects.

For beginning, we can take the typical case of permutations or of monic polynomials over finite fields. The latter deserves a special treatment in \cite{MullenPanario2013}. In \cite{PanRic_2001} and \cite{PanRic_2001_small_explog}, local results about the probability distribution of $X_{n}$ and asymptotic results about the $k$-th moment of $X_n$ are given. One of our goals in this paper is to revisit some results concerning the second moment in order to compute the variance of $X_n$, denoted by $\mathrm{Var}(X_{n})$. We recall that, by definition,
\begin{equation}
\mathrm{Var}(X_{n})=\sum_{k=1}^{n}{\big(k-\mathbf{E}(X_{n})\big)^{2}\phantom{\cdot}\mathbf{P}\{X_n = k\}}=\mathbf{E}(X_{n}^{2})-(\mathbf{E}(X_n))^{2},\label{var_defn}
\end{equation}
where $\mathbf{P}\{X_n = k\}$ is the probability that $X_n$ equals $k$, and $\mathbf{E}(X_n)$ is the expectation of $X_{n}$.

The $k$-th moments of $X_n$, that is $\mathbf{E}(X_{n}^{k})$, is expressed as an integral involving the ordinary Buchstab function $\omega$, which is defined over the real interval $[1,\infty)$ by
\begin{equation}
\omega(x)=\frac{1}{x}\quad\text{for $1\leq x\leq 2$}\quad\text{and}\quad \frac{\mathrm{d}(x\omega(x))}{\mathrm{d}x}=\omega(x-1)\quad\text{for $x\geq 2$}.\label{defn_btfct}
\end{equation}
In general as mentioned in \cite{PanRic_2001}, the $k$-th moment of $X_n$ involves the quantity $\int_{1}^{\infty}{t^{-k}\omega(t) \mathrm{d}t}$. Besides the original paper by Buchstab \cite{Buchstab_1937} in which the function is defined and analyzed, there are numerous other papers discussing its various properties and applications such as \cite{Bruijn_1950}. The book \cite{tenenbaum2018introduction} contains many useful properties about the Buchstab function as well as their proofs.

Theorem $5$ from \cite{PanRic_2001_small_explog} stipulates that
\begin{equation}
\mathrm{Var}(X_n)=C\big(n + O(n^{-\epsilon})\big)\quad\text{for some $\epsilon>0$}\label{defn_const}.
\end{equation}
The constant $C$ from (\ref{defn_const}) is given by
\begin{equation}
C = 2\int_{1}^{\infty}{\frac{\omega(t)}{t^{2}}\mathrm{d}t}.\label{C_integral_expr}
\end{equation}
\begin{remark}
We would like to point out that, in \cite{PanRic_1998_onBenOr}, \cite{PanRic_2001}, \cite{PanRic_2001_small_explog}, and also \cite{BenMasPanRic_2003}, the interval of integration in (\ref{C_integral_expr}) starts at $2$. The authors therein just forgot inadvertently to add $3/4$ resulting from the integration over the interval $[1,2)$ when computing the variance. This mistake lead to confusion of some researchers, see~\cite{Finch_2021_1}.
\end{remark}

Let $S_n$ be the set of permutations on $n$ elements, and let $S_{k,n}\subsetneq S_{n}$ be those permutations with smallest cycles of length $k$ for $1\leq k\leq n$. Denote the cardinality of $S_{k,n}$ by $s_{k,n}$. Let $c_{k}=(k-1)!$ for $k\geq 1$, and let $[n/k]=1$ if and only if $k|n$ otherwise $[n/k]=0$. Then, \cite{PanRic_2001} proves that
\begin{align}
s_{k,n}&=\sum_{i=1}^{\lfloor n/k\rfloor}{\frac{c_{k}^{i}}{i!}\frac{n!}{(k!)^{i}(n-ki)!}\sum_{j=k+1}^{n-ki}{s_{j,n-ki}}}+[n/k]\frac{c_{k}^{n/k}}{(n/k)!}\frac{n!}{(k!)^{n/k}}\label{in_paper_formula}\\
&=\sum_{i=1}^{\lfloor n/k\rfloor}{\frac{n!}{k^{i}i!(n-ki)!}\sum_{j=k+1}^{n-ki}{s_{j,n-ki}}}+[n/k]\frac{n!}{(n/k)!{}k^{n/k}}.\label{rec_lin_formula}
\end{align}
In order to simplify the notation from \cite{PanRic_2001} to fit our purpose here, we changed slightly the notation from $L_{k,n}^{s}$ to $s_{k,n}$.

For a fixed $n$, we have at least the following two properties:
\begin{displaymath}
s_{n,n}=(n-1)!,\quad s_{k,n}=0\text{ for $\lfloor n/2\rfloor+1 \leq k \leq n-1$,}\quad\text{and }\sum_{k=1}^{n}{s_{k,n}}=n!
\end{displaymath}
We have for a fixed $n\geq 1$ that
\begin{displaymath}
\mathbf{P}\{X_{n}=k\}=\frac{s_{k,n}}{n!}\quad\text{for $1\leq k\leq n$}.
\end{displaymath}

In Section \ref{sect_approaches}, we evaluate $C$ from (\ref{defn_const}) using different approaches. Another of our goals, pertaining to Section \ref{sect_gen_bt}, is to evaluate the generalized Buchstab\footnote{We thank an anonymous referee to have brought to our attention that the function considered here is not exactly a possible generalization of the original Buchstab because there is no $K$ such that $\Omega_{K}$ coincides with $\omega$ on the interval $[1,2)$.} function with parameter $K>0$ defined by
\begin{equation}
\Omega_{K}(x)=\left\{
\begin{array}{ll}
1 & \text{for $1\leq x <2$,}\\
1+K\int_{2}^{x}{\frac{\Omega_{K}(u-1)}{u-1}\mathrm{d}u} &\text{for $x\geq 2$.}
\end{array}
\right.\label{gen_Buchstab_def}
\end{equation}
The fraction of $n$-objects with large smallest components is given by $1/\Omega_{K}(x)$; more precisely, Theorem $1.1$ from \cite{BenMasPanRic_2003} stipulates that
\begin{equation*}
\lim_{n\to\infty}{\frac{s_{\lfloor xn\rfloor,\lfloor xn\rfloor}}{\sum_{i=n}^{\lfloor xn\rfloor}{s_{\lfloor xn\rfloor,i}}}}=\frac{1}{\Omega_{K}(x)}\quad\text{for $x > 1$}.
\end{equation*}

For the sake of completeness and to gain insight how the Buchstab function connects to combinatorial analysis, we end this introduction by recalling briefly how Buchstab introduced his function $\omega$ when studying the factorization of natural numbers into primes. The primes are like the irreducible factors of a polynomial, or the cycles of a permutation, etc. Let $\xi\in\{1,\ldots,n\}$ with its decomposition into primes given as $p_{1}(\xi)\cdots p_{k}(\xi)=\xi$ such that $p_{1}(\xi) \leq p_{2}(\xi) \leq \ldots \leq p_{k}(\xi)$. We count the number of $\xi$'s with their smallest prime factor less than $m$; in other words, set
\begin{displaymath}
\Psi(n,m)=\mathrm{card}\{\xi\in\{1,\ldots,n\}\colon p_{1}(\xi)\leq m\}.
\end{displaymath}
Then \cite{Buchstab_1937} showed that
\begin{displaymath}
\Psi(n,m)=1+\sum_{p\leq m}{\Psi\bigg(\frac{n}{p},p\bigg)}\quad\text{for all $1<m\leq n$}.
\end{displaymath}
The previous summation is over all primes $p$ less than or equal to $m$. The functional equation given $\Psi$ is connected to another important function, the Dickman function, that we do not discuss here; see \cite{tenenbaum2018introduction} for a detailed analysis of the Dickman function together with the Buchstab function.

\section{Approaches}\label{sect_approaches}

\subsection{Analytic estimation}\label{subsec_analytic_estimation}

In this section, we recall mostly results from \cite{PanRic_1998_onBenOr} and \cite{PanRic_2001_small_explog}. The approach from \cite{PanRic_2001_small_explog} to obtain the limiting quantities for $\mathbf{P}\{X_n\geq k\}$ and $\mathbf{E}(X_{n}^{\ell})$ as $k,n\to\infty$ and $\ell\geq 1$ uses singularity analysis of exponential generating functions for combinatorial objects. For an in-depth coverage of singularity analysis applied to combinatorics, see \cite{FlaSed_BOOK}.

Permutations form a typical class of combinatorial objects that we choose here for our discussion, but the results are not limited only to permutations. The cycles are seen as the irreducible components of a permutation. Let $C(z)=\sum_{i=0}^{\infty}{C_{i}z^{i}/i!}$ be the exponential generating function for counting cycles of given lengths. Then the exponential generating function for counting permutations of given sizes is
\begin{equation*}
L(z)=\exp(C(z))=\sum_{i=0}^{\infty}{L_{i}\frac{z^{i}}{i!}}.
\end{equation*}
For a fixed $n>0$, we are interested in counting permutations with smallest cycles of length at least $k$ for $1\leq k\leq n$. Let $S(z)$ be the generating function for counting permutations with smallest cycles of length at least $k$ for $1\leq k\leq n$. Then we have
\begin{equation*}
S(z)=\exp\bigg(\sum_{i=k}^{\infty}{C_{i}\frac{z^{i}}{i!}}\bigg)-1=\sum_{i=0}^{\infty}{S_{i}\frac{z^{i}}{i!}}
\end{equation*}
Therefore the tail of the probability distribution of $X_n$ is given by
\begin{equation*}
\mathbf{P}\{X_n\geq k\}=\frac{S_{n}}{L_{n}}.
\end{equation*}
Using singularity analysis, \cite{PanRic_2001_small_explog} shows that if $k,n\to\infty$, then
\begin{equation}
\mathbf{P}\{X_n\geq k\}=\frac{1}{k}\omega\bigg(\frac{n}{k}\bigg)+O\bigg(\frac{1}{k^{1+\epsilon}}\bigg)\quad\text{for some $\epsilon>0$.}\label{local_prob_res}
\end{equation}
Theorem \ref{high_order_mom_prop} states the asymptotic behaviour of the moments.
\begin{theorem}\label{high_order_mom_prop}
For some function $h(n)$ that tends slower to infinity than $\log(n)$ and for some $\epsilon>0$ independent of $n$, we have that
\begin{align*}
\mathbf{E}(X_{n})&=e^{-\gamma}\log(n)\bigg(1+O\bigg(\frac{h(n)}{\log(n)}\bigg)\bigg),\\
\mathbf{E}(X_{n}^{\ell})&=\ell{}n^{\ell-1}\bigg(\int_{1}^{\infty}{\frac{\omega(x)}{x^{\ell}}\mathrm{d}x}\bigg)\bigg(1+O\bigg(\frac{1}{n^{\epsilon}}\bigg)\bigg)\quad\text{for integer $\ell\geq 2$.}
\end{align*}
\end{theorem}
\begin{proof}
We consider the case when $\ell\geq 2$. We give the main steps for the proof of Theorem \ref{high_order_mom_prop}. By definition, we have
\begin{equation}
\mathbf{E}(X_{n}^{\ell})=\sum_{k=1}^{\infty}{\big(k^{\ell}-(k-1)^{\ell}\big)\mathrm{P}\{X_{n}\geq k\}}.\label{sum_for_high_order_mom}
\end{equation} Let $\nu(n)=\lfloor n^{\epsilon'}\rfloor$ such that $0<\epsilon'<\epsilon$ where $\epsilon$ is given from (\ref{local_prob_res}). Then $\nu(n)\to\infty$ as $n\to\infty$, we split the sum from (\ref{sum_for_high_order_mom}) using $\nu$, and we obtain
\begin{align*}
\mathbf{E}(X_{n}^{\ell})&=\sum_{k=1}^{\nu(n)-1}{\big(k^{\ell}-(k-1)^{\ell}\big)\mathrm{P}\{X_{n}\geq k\}}+\sum_{k=\nu(n)}^{\infty}{\big(k^{\ell}-(k-1)^{\ell}\big)\mathrm{P}\{X_{n}\geq k\}}\\
&\isdef S_{1}+S_{2}.
\end{align*}
Using (\ref{local_prob_res}), and the fact that $\mathbf{P}\{X_{n}\geq n+1\}=0$, we have $S_{1}=O\big((\nu(n))^{\ell-1}\big)$ because $(k^{\ell}-(k-1)^{\ell})\in O(k^{\ell-1})$ and $\mathrm{P}\{X_{n}\geq k\}\in O(1/k^{1+\epsilon})$ in the range $1\leq k <\nu(n)$. In the range $\nu(n)\leq k \leq n$, we have $(k^{\ell}-(k-1)^{\ell})\in O(\ell{}k^{\ell-1})$, and therefore
\begin{align}
S_{2}&=\sum_{k=\nu(n)}^{\infty}{\big(k^{\ell}-(k-1)^{\ell}\big)\mathrm{P}\{X_{n}\geq k\}}\nonumber\\
&=\ell\Bigg(\sum_{k=\nu(n)}^{n}{k^{\ell-2}\omega\bigg(\frac{n}{k}\bigg)}\Bigg)\big(1+O(\nu(n)^{-\epsilon})\big).\label{eq_S2}
\end{align}
The sum within (\ref{eq_S2}) is viewed as a Riemann sum that is estimated by its corresponding integral
\begin{align*}
\sum_{k=\nu(n)}^{n}{k^{\ell-2}\omega\bigg(\frac{n}{k}\bigg)} &= \int_{0}^{n}{t^{\ell-2}\omega\bigg(\frac{n}{t}\bigg)\mathrm{d}t}+O\bigg(\frac{1}{n}\bigg)\\
&=n^{\ell-1}\int_{1}^{\infty}{\frac{\omega(x)}{x^{\ell}}\mathrm{d}x}+O\bigg(\frac{1}{n}\bigg)\quad\text{with $\frac{n}{t} = x$.}
\end{align*}
The proof for the case $\ell=1$ is quite similar, and the range $\nu(n) \leq k\leq n$ is simply divided further into two ranges $\nu(n)\leq k < n\mu(u)$ and $n\mu(n)\leq k\leq n$ where $\mu(n)$ for some well-chosen function $\mu$ as in \cite{PanRic_1998_onBenOr}.
\end{proof}

\begin{remark}
The sum in (\ref{eq_S2}) goes up to $n$ inclusively and not $n/2$; thus the range of integration starts at $1$ and not $2$. Because $\mathbf{P}\{X_n = k\}=0$ for $\lfloor n/2\rfloor + 1\leq k\leq n-1$, we point out as well that
\begin{displaymath}
\mathbf{P}\{X_{n}\geq k\}=\sum_{i=k}^{n}{\mathbf{P}\{X_{n}=i\}}=\mathbf{P}\{X_{n}=n\}\quad\text{for $\lfloor n/2\rfloor + 1\leq k\leq n$}.
\end{displaymath}
\end{remark}

Back to the variance of $X_n$, we have the following theorem that ends our section on the analytical estimation for $\mathrm{Var}(X_{n})/n$ as $n\to\infty$.
\begin{theorem}
For some $\epsilon>0$ independent of $n$, we have that
\begin{displaymath}
\mathrm{Var}(X_{n})=nC\bigg(1+O\bigg(\frac{1}{n^{\epsilon}}\bigg)\bigg)\quad\text{with}\quad C=2\int_{1}^{\infty}{\frac{\omega(x)}{x^{2}}\mathrm{d}x}
\end{displaymath}
\end{theorem}
\begin{proof}
We have by definition that $\mathrm{Var}(X_{n})=\mathbf{E}(X_{n}^{2})-\big(\mathbf{E}(X_{n})\big)^{2}$. We use (\ref{local_prob_res}) and consider the second moment. Hence we have
\begin{align}
\mathbf{E}(X_{n}^{2})&=\sum_{k=1}^{\infty}{\big(k^{2}-(k-1)^{2}\big)\mathbf{P}\{X_{n}\geq k\}}=\sum_{k=1}^{\infty}{\big(2k-1\big)\mathbf{P}\{X_{n}\geq k\}}\nonumber\\
&= \sum_{k=1}^{n}{\big(2k-1\big)\mathbf{P}\{X_{n}\geq k\}}\nonumber\\
&= \sum_{k=1}^{n}{\bigg(2k-1\bigg)\bigg(\frac{1}{k}\omega\bigg(\frac{n}{k}\bigg)+O\bigg(\frac{1}{k^{1+\epsilon}}\bigg)\bigg)}\quad\text{for some $\epsilon>0$}\nonumber\\
&\sim 2\sum_{k=1}^{n}{\omega\bigg(\frac{n}{k}\bigg)}.\label{var_asympto}
\end{align}
The expression (\ref{var_asympto}) is a Riemann sum and is estimated in a similar way as in Proposition \ref{high_order_mom_prop}. The quantity $\big(\mathbf{E}(X_{n})\big)^{2}$ is negligible compared to $\mathbf{E}(X_{n}^{2})$ as $n\to \infty$. Hence we have that
\begin{displaymath}
\mathrm{Var}(X_{n})\sim 2n\int_{1}^{\infty}{\frac{\omega(x)}{x^{2}}\mathrm{d}x}\quad\text{as $n\to\infty$.}
\end{displaymath}

In \cite{Selberg_1944}, it is shown that $\omega(x)\to e^{-\gamma}$ where $\gamma$ is the Euler-Mascheroni constant. More specifically, it was shown that $|\omega(x)-e^{-\gamma}|<10^{-4}$ for $x>4$. Therefore we have that
\begin{displaymath}
C=2\int_{1}^{\infty}{\frac{\omega(x)}{x^{2}}\mathrm{d}x}=2\int_{1}^{4}{\frac{\omega(x)}{x^{2}}\mathrm{d}x}+2\int_{4}^{\infty}{\frac{e^{-\gamma}}{x^{2}}\mathrm{d}x}+2\int_{4}^{\infty}{\frac{\omega(x)-e^{-\gamma}}{x^{2}}\mathrm{d}x}.
\end{displaymath}
Using the quantities from \cite{PanRic_1998_onBenOr} for
\begin{displaymath}
2\int_{2}^{\infty}{\frac{\omega(x)}{x^{2}}\mathrm{d}x}=0.5586\ldots,
\end{displaymath}
\emph{and}, this time, taking into account the evaluation of the integral over $[1,2]$ that yields exactly $3/4$, we obtain up to four significant figures that $C=1.3068\ldots$, and thus
\begin{equation*}
\frac{\mathrm{Var}(X_{n})}{n}\to 1.3068\ldots\quad\text{as $n\to\infty$.}
\end{equation*}
The proof is now complete.
\end{proof}

\subsection{Numerical integration}\label{subsec_num_int}

We adapt an idea from \cite{MarZamMar_1989} in Theorem \ref{prop_marsa} to evaluate with an arbitrary finite precision $\omega(x)$ for any $x\geq 1$. We use Theorem \ref{prop_marsa} to evaluate $C$. The quantity $n$ in this section is not the same as previously that stands for the number of elements considered in our combinatorial object while $n$ here stands for the integral part of a real number, as it is standard in numerical approximations.

We recall that we need to evaluate
\begin{equation}
C=2\int_{1}^{\infty}{\frac{\omega(t)}{t^{2}}\mathrm{d}t}=\lim_{n\to\infty}{\frac{\mathrm{Var}(X_n)}{n}}.\label{defn_2nd_moment}
\end{equation}

For notational simplicity, we use $f:[1,\infty)\to [0,1]$ to denote the function $x\mapsto \omega(x)/x^{2}$. As mentioned previously, $|\omega(x)-e^{-\gamma}|<10^{-4}$ for $x>4$, then $f$ is bounded. The function $f$ is also continuous because it is the composition of two continuous functions on $[1,\infty)$. We have that $f(x)\to 0$ as $x\to \infty$. Hence the Riemann sum of $f$ is convergent. We can approximate numerically its Riemann sum, that is $\int_{1}^{\infty}{f(t)\mathrm{d}t}$, up to a desired accuracy by truncating the integral; this is justified by the fact that $f(x)\to 0$.

A popular method to approximate an integral is the trapezoidal method with a regular grid of points. Consider the interval $[1,n^{\ast}]$ where $n^{\ast}\in\mathbb{N}$ shall be determined later. Given the nature of $\omega$ (and so $f$), we consider for now an interval of the form $[n,n+1]$ where $n\in\mathbb{N}$. A point from a regular grid on $[n,n+1]$ can be put conveniently into the form $x_{i}=n+i\delta$ for $0\leq i \leq \ell$ where $\delta=2^{-\ell}$. We therefore have that
\begin{equation}
\sum_{i=0}^{2^{\ell}-1}{\delta\frac{\big(f(n+i\delta)+f(n+(i+1)\delta)\big)}{2}} \to \int_{n}^{n+1}{f(t)\mathrm{d}t}\quad\text{as $\ell \to \infty$}.\label{zozo_rabou}
\end{equation}

To evaluate $C$ with four significant digits, we can select $n^{\ast}=10000$ and $\ell=14$ so that $\delta<10^{-4}$ using for instance the sharp bounds on numerical integration from \cite{CruNeu_2002}. Now it remains to know how to compute numerically $\omega(x)$ for $x\geq 1$, which is done using Taylor series as given by Theorem \ref{prop_marsa}.

\begin{theorem}\label{prop_marsa}
Consider the Taylor expansions of $\omega$ with respect to the $z$ variable for each unit length interval of the form $[n,n+1)$. More precisely let
\begin{displaymath}
\omega\bigg(n+\frac{1+z}{2}\bigg) = \sum_{i=0}^{\infty}{c_{n,i}z^{i}}\quad\text{for $n\geq 1$ and for $-1\leq z < 1$}.
\end{displaymath}
Let $c_{n,i}$ the $i$-th term for $n$-th sequence $\mathbf{c}_{n}$ for $n\geq 1$ and $i\geq 0$. Then we have
\begin{align*}
c_{1,i}&=\frac{2}{3}\bigg(\frac{-1}{3}\bigg)^{i}\quad\text{for $i\geq 0$,}\\
c_{n+1,0}&=\frac{1}{2n+3}\sum_{i=0}^{\infty}{c_{n,i}\bigg(2(n+1)+\frac{(-1)^{i}}{i+1}\bigg)}\quad\text{for $n > 1$,}\\
c_{n+1,i}&=\frac{1}{2n+3}\bigg(\frac{c_{n,i}}{n}-c_{n+1,i-1}\bigg)\quad\text{for $n > 1$ and $i\geq 1$.}
\end{align*}
\end{theorem}
\begin{proof}
Let $n\geq 1$ and let $x=n+t\geq 1$ with $n=\lfloor x\rfloor$ and $0\leq t < 1$. If $\omega$ has a Taylor expansion in $[n,n+1)$, that is the coefficients $c_{n,i}$, then we obtain the coefficients $c_{n+1,i}$ of the Taylor expansion in $[n+1,n+2)$ as follows. We integrate the difference-differential equation (\ref{defn_btfct}) and have that
\begin{align*}
\int_{u=n+1}^{u=n+1+t}{\mathrm{d}(u\omega(u))}&=(n+1+t)\omega(n+1+t)-(n+1)\omega(n+1)\\
&=\int_{u=n+1}^{u=n+1+t}{\omega(u-1)\mathrm{d}u}\\
%&=\int_{u-(n+1)=0}^{u-(n+1)=t}{\omega\big((u-(n+1))+n\big)\mathrm{d}(u-(n+1))}\\
&=\int_{x=0}^{x=t}{\omega(n+x)\mathrm{d}x},\quad\text{with $u= n+1+x$.}
\end{align*}
The affine transformation $t= z=2t+1$ transforms the fractional part $t\in[0,1)$ into a centered-around-$0$ value $z\in[-1,1)$. Equivalently $t=(z+1)/2$, and therefore we have that
\begin{align}
&\Big(n+1+\frac{z+1}{2}\Big)\omega\Big(n+1+\frac{z+1}{2}\Big) - (n+1)\omega(n+1)\label{int_before_affine}\\
\qquad=&\int_{v=0}^{v=(z+1)/2}{\omega(n+1+v)\mathrm{d}v}\nonumber\\%\int_{(u+1)/2= 0}^{(u+1)/2=(z+1)/2}{\omega\Big(n+\frac{u+1}{2}\Big)\mathrm{d}\Big(\frac{u+1}{2}\Big)}\nonumber\\
\qquad=&\frac{1}{2}\int_{u=-1}^{u=z}{\omega\Big(n+\frac{u+1}{2}\Big)\mathrm{d}u}\quad\text{with $v= \frac{u+1}{2}$}\label{int_after_affine}.
\end{align}
Using Taylor expansion around $u=0$ of $\omega$ in the interval $[n,n+1)$ in terms of the dummy variable of integration, we have
\begin{align}
\omega\Big(n+\frac{u+1}{2}\Big)&=\sum_{i=0}^{\infty}{c_{n,i}u^{i}}\quad\text{for $-1\leq u \leq z < 1$}.\label{expand_around_0_in_n_np1}
\end{align}
Hence by substituting (\ref{expand_around_0_in_n_np1}) into (\ref{int_after_affine}):
\begin{equation}
\int_{u=-1}^{u=z}{\omega\Big(n+\frac{u+1}{2}\Big)\mathrm{d}u}=\int_{-1}^{z}{\sum_{i=0}^{\infty}{c_{n,i}u^{i}}\mathrm{d}u}=\sum_{i=0}^{\infty}{c_{n,i}\frac{\big(z^{i+1}-(-1)^{i+1}\big)}{i+1}}.\label{integration_dummy_u}\\
\end{equation}
By continuity of $\omega$, we have also that
\begin{equation}
\lim_{z\to 1}\omega\Big(n+\frac{z+1}{2}\Big)=\omega(n+1)=\lim_{z\to 1}{\sum_{i=0}^{\infty}{c_{n,i}z^{i}}}=\sum_{i=0}^{\infty}{c_{n,i}}.\label{limit_and_continuity_omega}
\end{equation}
Using Taylor expansion around $z=0$ of $\omega$ in the interval $[n+1,n+2)$, we obtain
\begin{displaymath}
\omega\Big(n+1+\frac{z+1}{2}\Big) = \sum_{i=0}^{\infty}{c_{n+1,i}z^{i}}\quad\text{for $-1\leq z < 1$}.
\end{displaymath}
Then substituting (\ref{limit_and_continuity_omega}) into (\ref{int_before_affine}), equating $1/2$ times (\ref{integration_dummy_u}) to (\ref{int_after_affine}), and multiplying by $2$ both sides of the equality yields:
\begin{align}
(2n+3+z)\sum_{i=0}^{\infty}{c_{n+1,i}z^{i}}&=2(n+1)\sum_{i=0}^{\infty}{c_{n,i}}+\sum_{i=0}^{\infty}{c_{n,i}\frac{\big(z^{i+1}-(-1)^{i+1}\big)}{i+1}}.\label{eqn_II}
\end{align}
Substituting $z=0$ in (\ref{eqn_II}), we get
\begin{align}
c_{n+1,0}&=\frac{1}{2n+3}\sum_{i=0}^{\infty}{c_{n,i}\bigg(2(n+1)+\frac{(-1)^{i}}{i+1}\bigg)}.\label{b0_coef}
\end{align}
By using (\ref{b0_coef}) and gathering equal-like powers of $z$, we find $c_{n+1,i}$ for $i\geq 1$ as follows:
\begin{eqnarray*}
  && (2n+3+z)c_{n+1,0}+(2n+3+z)\sum_{i=1}^{\infty}{c_{n+1,i}z^{i}}\\
&=& 2(n+1)\sum_{i=0}^{\infty}{c_{n,i}}
     +\sum_{i=0}^{\infty}{c_{n,i}\frac{\big(z^{i+1}+(-1)^{i}\big)}{i+1}},
\end{eqnarray*}
\begin{align*}
c_{n+1,0}z+(2n+3+z)\sum_{i=1}^{\infty}{c_{n+1,i}z^{i}}&=c_{n,0}z+\sum_{i=1}^{\infty}{c_{n,i}\frac{z^{i+1}}{i+1}}\quad\text{, and}
\end{align*}
\begin{eqnarray*}
  && (2n+3+z)\sum_{i=1}^{\infty}{c_{n+1,i}z^{i}}\\
&=& (2n+3)c_{n+1,1}z+(2n+3)\sum_{i=2}^{\infty}{c_{n+1,i}z^{i}}
     +\sum_{i=1}^{\infty}{c_{n+1,i}z^{i+1}}.
\end{eqnarray*}
The previous equation holds if and only if
\begin{align*}
\big((2n+3)c_{n+1,i}+c_{n+1,i-1}\big)z^{i}&=\frac{c_{n,i-1}z^{i}}{i}\quad\text{for all $i\geq 1$.}
\end{align*}
We finally find the Taylor expansion $1/x$ around $x=1$ with $1\leq x=1+t\leq 2$ and $t=(1+z)/2$ for $-1\leq z < 1$, and have
\begin{align*}
\omega\bigg(1+\frac{1+z}{2}\bigg)&=\frac{2}{3}\frac{1}{(1+(z/3))}=\frac{2}{3}\sum_{i=0}^{\infty}{\bigg(\frac{-1}{3}\bigg)^{i}z^{i}}=\sum_{i=0}^{\infty}{c_{1,i}z^{i}}.
%&c_{1,i}&=\frac{2}{3}\bigg(\frac{-1}{3}\bigg)^{i}.
\end{align*}
The proof is now complete.
\end{proof}
We point out that the centered-around-0 flavour of the Taylor expansions $\mathbf{c}_{n}$ allows faster convergence around the endpoints $n$ and $n+1$, see \cite{MarZamMar_1989}. We compute the first $n^{\ast}$ sequences with their first $J$ terms, say, and provided we have a library that does real arithmetic with a finite and arbitrary precision.
\begin{breakablealgorithm}\label{algo_trapezo}
\caption{Trapezoidal rule by using Taylor coefficient of the Buchstab function on the interval $[n,n+1)$ for $n\in\mathbb{N}$}
\begin{algorithmic}[1]
\raggedright
\Require $\ell$, $n$, $\{c_{n,j}\}_{j=0}^{J}$
\Ensure s, the sum from \ref{zozo_rabou}.
\State{$\delta \leftarrow 2^{-\ell}$}
\State{$s \leftarrow 0$}
\For{$i=0$ \textbf{to} $2^{\ell}-1$}
\State{$y_{0} \leftarrow 0$}
\State{$y_{1} \leftarrow 1$}
\State{$t_{0} \leftarrow i\delta$}
\State{$t_{1} \leftarrow (i+1)\delta$}\label{alg_trap_c1}
\State{$z_{0} \leftarrow 1$}
\State{$z_{1} \leftarrow 1$}
\For{$j=0$ \textbf{to} $J$}\label{algo_trap_c2}
\State{$y_{0} \leftarrow y_{0} + c_{n,j}z_{0}$}\label{algo_trap_c3}
\State{$y_{1} \leftarrow y_{1} + c_{n,j}z_{1}$}\label{algo_trap_c4}
\State{$z_{0} \leftarrow z_{0}(2t_{0}-1)$}\label{algo_trap_c5}
\State{$z_{1} \leftarrow z_{1}(2t_{1}-1)$}\label{algo_trap_c6}
\EndFor
\State{$s \leftarrow s + \frac{y_0}{(n+t_{0})^2} + \frac{y_1}{(n+t_{1})^2}$}\label{algo_trap_c7}
\EndFor
\State{$s \leftarrow \frac{s\delta}{2}$}\label{algo_trap_c8}
\end{algorithmic}
\end{breakablealgorithm}

To obtain $C$, we call iteratively Algorithm \ref{algo_trapezo} for values of $n=1,2,\ldots, n^{\ast}$ with the coefficients for the Taylor expansion of $\omega$ on the interval $[n,n+1)$. We add the result of all iterations together and obtain $C=1.3070\ldots$, which confirms comfortably the estimation from Section \ref{subsec_analytic_estimation}.

We end this section with a few comments about Algorithm \ref{algo_trapezo}. We have in line (\ref{alg_trap_c1}) that $t_{1} = t_{0}+\delta$. The loop at line (\ref{algo_trap_c2}) computes the Taylor polynomial of degree $J$ of the Buchstab function $\omega(n+(1+z)/2)$ for the specific values of $z=z_0$, and $z=z_1$. During the $j$-th iteration at the lines (\ref{algo_trap_c3}) and (\ref{algo_trap_c4}), we have that $y_{b}=\sum_{k=0}^{j}{c_{n,k}z_{b}^{k}}$ for $b=0$ and $b=1$, respectively. Lines (\ref{algo_trap_c5}) and (\ref{algo_trap_c6}) are for updating respectively $z_0$ and $z_1$ for the next iteration, that is, the $(j+1)$-th iteration. We recall the meaning of the left side of the limiting expression (\ref{zozo_rabou}) is that the height of a rectangle is $\big(f(n+i\delta)+f(n+(i+1)\delta\big)/2$ with $f(x)=\omega(x)/x^{2}$ in our case, and its length $\delta$; therefore line (\ref{algo_trap_c7}) sums over the heights of all the rectangles. Averaging two consecutive heights by $2$ is carried out only once at line (\ref{algo_trap_c8}) so that we save a few operations. Similarly, we take into account the length $\delta$, which is identical for each rectangle, only once at line (\ref{algo_trap_c8}).

\subsection{Recurrence relation}\label{subsec_rec_reln}

We compute the probability distribution of $X_{n}$ and then compute $\mathrm{Var}(X_{n})$ for values of $n=1,2,\ldots, 4000$. Recalling (\ref{var_defn}), we have that
\begin{displaymath}
\mathrm{Var}(X_{n})=\sum_{k=1}^{n}{\big(k-\mathbf{E}(X_{n})\big)^{2}\phantom{\cdot}\mathbf{P}\{X_n = k\}}.
\end{displaymath}
Because
\begin{displaymath}
\mathbf{E}(X_{n})=\sum_{k=1}^{n}{k\mathbf{P}\{X_n = k\}}\quad\text{and}\quad\mathbf{P}\{X_n = k\}=\frac{s_{k,n}}{n!},
\end{displaymath}
the variance can therefore be expressed as a rational number, which is suitable to control the accuracy, as follows:
\begin{displaymath}
\frac{n!\sum_{k=1}^{n}{k^{2}s_{n,k}}-\Big(\sum_{k=1}^{n}{ks_{n,k}}\Big)^{2}}{(n!)^{2}}.
\end{displaymath}
We divide the quantity $\mathrm{Var}(X_{n})$ by $n$ in order to normalize. We recall that $\mathrm{Var}(X_{n})=C(n+O(n^{-\epsilon}))$ for some $\epsilon>0$. When computing exactly $\mathrm{Var}(X_{n})$ for a fixed $n$ and comparing with the asymptotic formula, one would need the hidden factor of $n^{-\epsilon}$ and the value $\epsilon$ itself in order make a fair comparison; we nevertheless obtain numbers that are very close to the numbers from Sections \ref{subsec_analytic_estimation} and \ref{subsec_num_int}.

\begin{align*}
&\frac{\mathrm{Var}(X_{1000})}{1000}=1.3004\ldots,\quad \frac{\mathrm{Var}(X_{2000})}{2000}=1.3036\ldots,\\
&\frac{\mathrm{Var}(X_{3000})}{3000}=1.3047\ldots,\quad \frac{\mathrm{Var}(X_{4000})}{4000}=1.3053\ldots.
\end{align*}

The size of the memory on the machines available to us is the main limitation here; however it is enough to assert $C$ up to two significant digits. A space of $12.7GB$ is needed to compute the triangular table for $n=4000$. The recurrence relation is easily computed by storing the values into a triangular array. We observe that is very hard to trim the array of potentially unused cells as $n$ grows. Each cell of the array holds $s_{n,k}$ for a pair $(n,k)$. The values $s_{n,k}$ are given by (\ref{rec_lin_formula}). We could compress the array slightly for $s_{n,k}$ when $\lfloor n/2 \rfloor+1 \leq k \leq n-1$ using methods described in \cite{Navarro_book_2016} for instance, but we would not gain much for large values of $n$ (like $n>1000$) in space and would yield a more complicated code.

A possible algorithm for counting the $s_{n,k}$ is as in Algorithm \ref{alg_rec_reln}.
\begin{breakablealgorithm}\label{alg_rec_reln}
\caption{Computing $s_{n,k}$}
\begin{algorithmic}[1]
\raggedright
\Require $N$
\Ensure $s_{n,k}$ for $1\leq n\leq N$ and $1\leq k\leq n$
\State{$s_{0,0} \leftarrow 1$}\label{algo_rec_lin_c1}
\For{$n=1$ \textbf{to} $N$}
\State{$s_{n,0} \leftarrow 0$}\label{algo_rec_lin_c2}
\State{$s_{n,n} \leftarrow (n-1)!$}
\EndFor
\For{$n=2$ \textbf{to} $N$}
\For{$k=1$ \textbf{to} $\lfloor n/2\rfloor$}\label{algo_rec_lin_c3}
\State{$t_{1} \leftarrow 0$}
\For{$i=1$ \textbf{to} $\lfloor n/k\rfloor$}
\State{$u_{1} \leftarrow 0$}
\For{$j=k+1$ \textbf{to} $n-ki$}
\State{$u_{1} \leftarrow u_{1}+s_{n-ki,j}$}
\EndFor
\If{$k+1 \leq n-ki$}
\State{$u_{1}\leftarrow u_{1}\frac{n!}{i!{}k^{i}(n-ki)!}$}
\EndIf
\State{$t_{1}\leftarrow t_{1}+u_{1}$}
\EndFor
\State{$t_{2} \leftarrow 0$}
\If{$k$ divides $n$}
\State{$t_{2}\leftarrow \frac{n!}{(n/k)!{}k^{n/k}}$}
\EndIf
\State{$s_{n,k}\leftarrow t_1+t_2$}
\EndFor
\EndFor
\end{algorithmic}
\end{breakablealgorithm}

We make just a few comments about Algorithm \ref{alg_rec_reln}, from a data structure point of view, $n=0$ and $k=0$ are boundaries for the table and lines (\ref{algo_rec_lin_c1}) and (\ref{algo_rec_lin_c2}) define the programming boundaries, but are not part of the combinatorial objects and their related probability distributions a fortiori. The loop at line (\ref{algo_rec_lin_c3}) runs up to $\lfloor n/2\rfloor$ because it is assumed that $s_{n,k}$ are initialized to $0$ by default for all valid $n$ and $k$; this is usually the case in most advanced programming languages when declaring data structures.

We end this section with a small example. Table \ref{tab_rec_lin_example} shows $s_{n,k}$ for $1\leq n\leq 10$. We apologize for the font size that has to be changed temporarily in order to display the table.\vspace{1mm}

\setlength{\tabcolsep}{5pt}
\scriptsize
\begin{longtable}{l||rrrrrrrrrr}
\multicolumn{11}{c}{\normalsize Table \label{tab_rec_lin_example}\NewTabulare{}: Values of $s_{n,k}$ for $1\leq n\leq 10$.\scriptsize}\\ %\caption{Values of $s_{n,k}$ for $1\leq n\leq 10$}\label{tab_rec_lin_example}\\
\cline{2-11}
\multicolumn{1}{l}{} &\multicolumn{10}{|c}{$k$}\\
\hline
$n$  &         1&        2&        3&        4&        5&        6&        7&        8&        9&       10\\
\endfirsthead
\cline{2-11}
\multicolumn{1}{l}{} &\multicolumn{10}{|c}{$k$}\\
\hline
$n$  &         1&        2&        3&        4&        5&        6&        7&        8&        9&       10\\
\endhead
\hline \multicolumn{11}{c}{Continued on next page}
\endfoot
\endlastfoot
\hline
10&   2293839&   525105&   223200&   151200&    72576&        0&        0&        0&        0&   362880\\
 9&    229384&    52632&    22400&    18144&        0&        0&        0&        0&    40320&\\
 8&     25487&     5845&     2688&     1260&        0&        0&        0&     5040&&\\
 7&      3186&      714&      420&        0&        0&        0&      720&&&\\
 6&       455&      105&       40&        0&        0&      120&&&&\\
 5&        76&       20&        0&        0&       24&&&&&\\
 4&        15&        3&        0&        6&&&&&&\\
 3&         4&        0&        2&&&&&&&\\
 2&         1&        1&&&&&&&&\\
 1&         1&&&&&&&&&\\
\end{longtable}
\normalsize

\section{Generalized Buchstab function}\label{sect_gen_bt}

We recall (\ref{gen_Buchstab_def}), the definition of the generalized Buchstab function with parameter $K>0$, which is
\begin{equation}
\Omega_{K}(x)=\left\{
\begin{array}{ll}
1 & \text{for $1\leq x < 2$,}\\
1 + K\int_{2}^{x}{\frac{\Omega_{K}(u-1)}{u-1}\mathrm{d}u} & \text{for $x\geq 2$.}
\end{array}\right.
\end{equation}

Values of $1/\Omega_{K}(x)$ are asymptotic proportions of large smallest component as proved in \cite{BenMasPanRic_2003}. More precisely, we recall that $s_{n,k}$, given as in (\ref{in_paper_formula}) of Section \ref{sect_intro}, is the number of combinatorial $n$-objects with their smallest components having length $k$. For instance, the parameter $K=1/2$ includes $2$-regular graphs, surjective maps, etc. The parameter $K=1$ includes derangements, permutations, monic polynomials over a finite field, and so on. The quantity $\sum_{i=k}^{n}{s_{n,i}}$ is the number of $n$-objects for which the smallest component has size at least $k$ for $1\leq k\leq n$. Let $x>1$ and consider the ratio
\begin{equation}
\frac{s_{\lfloor xn\rfloor,\lfloor xn\rfloor}}{\sum_{i=n}^{\lfloor xn\rfloor}{s_{\lfloor xn\rfloor,i}}}.\label{ratio_limit}
\end{equation}
Then it is shown in \cite{BenMasPanRic_2003} that, for $x>1$,
\begin{equation}
\lim_{n\to\infty}{\frac{s_{\lfloor xn\rfloor,\lfloor xn\rfloor}}{\sum_{i=n}^{\lfloor xn\rfloor}{s_{\lfloor xn\rfloor,i}}}}=\frac{1}{\Omega_{K}(x)}\label{inverse_omega_K}.
\end{equation}
The limiting quantity (\ref{inverse_omega_K}) justifies our interests in evaluating the generalized Buchstab function.

We remark that from now on and up to Table \ref{tab_OmegaK} inclusively, the symbol $n$ does no longer refer to the size of a combinatorial object.

Following the ideas exposed in Section \ref{subsec_num_int}, let $n\geq 1$ be a natural number, and let $c_{n,i}$ be $i$-th coefficient of the Taylor expansion for $\Omega_{K}(z)$ in the interval $[n,n+1)$ with $1\leq z<1$. More precisely, let
\begin{equation}
\Omega_{K}\Big(n+\frac{1+z}{2}\Big)=\sum_{i=0}^{\infty}{c_{n,i}z^{i}}\quad\text{for $-1\leq z<1$.}\label{croco_boulette}
\end{equation}
As we might expect, the sequence $(c_{n,i})_{i\geq 0}$ depends on the previous sequence $(c_{n-1,i})_{i\geq 0}$ for $n>2$. Our library can compute with arbitrary finite precision over $\mathbb{R}$. The variable $z$ in (\ref{croco_boulette}) is the fractional part of $x\in[n,n+1)$ centered around $0$.

\begin{theorem}
For $K>0$, consider the Taylor expansions of $\Omega_{K}$ with respect to the $z$ variable for each unit length interval of the form $[n,n+1)$. More precisely, let
\begin{displaymath}
\Omega_{K}\bigg(n+\frac{1+z}{2}\bigg) = \sum_{i=0}^{\infty}{c_{n,i}z^{i}}\quad\text{for $n\geq 1$ and for $-1\leq z < 1$}.
\end{displaymath}
For $n\geq 1$ and $i\geq 0$, and let $\alpha_{i}$ be defined by
\begin{displaymath}
\alpha_{i}=\sum_{j=0}^{i}{\frac{(-1)^{i-j}}{(2n-1)^{i-j}}c_{n-1,j}}\quad\text{for $i\geq 0$.}
\end{displaymath}
Then we have
\begin{align*}
c_{1,0}&=1,\\
c_{1,i}&=0\quad\text{for $i\geq 1$},\\
c_{2,0}&=c_{2,0}=1+K\sum_{i=1}^{\infty}{\frac{1}{i2^{i}}},\\
c_{2,i}&=K\sum_{j=i}^{\infty}{\frac{(-1)^{j-1}}{j2^{j}}\binom{j}{i}}\quad\text{for $i\geq 1$},\\
c_{n,0}&=\sum_{i=0}^{\infty}{c_{n-1,i}}-\frac{K}{2n-1}\sum_{i=0}^{\infty}{\frac{(-1)^{i+1}\alpha_{i}}{i+1}}\quad\text{for $n \geq 3$,}\\
c_{n,i}&=\frac{K\alpha_{i-1}}{(2n-1)i}\quad\text{for $n \geq 3$ and $i\geq 1$.}
\end{align*}
\end{theorem}
\begin{proof}
For $x\in[1,2)$, the function $\Omega_{K}$ is constant and then $c_{1,0}=1$ and $c_{1,i}=0$ for $i\geq 1$.

For $2\leq x = 2+((1+z)/2) < 3$, the coefficients of the Taylor expansion are $1+K\log(2+(1+z)/2)$; hence the coefficients are given by
\begin{equation}
c_{2,0}=1+K\sum_{i=1}^{\infty}{\frac{1}{i2^{i}}}\quad\text{and}\quad c_{2,i}=K\sum_{j=i}^{\infty}{\frac{(-1)^{j-1}}{j2^{j}}\binom{j}{i}}\quad\text{for $i\geq 1$.}\label{coeff_int_2_3}
\end{equation}
Given $x\geq 3$ such that $x = n+((z+1)/2)$ so that $n\geq 3$ as well, we assume known the sequence $(c_{n-1,i})_{i\geq 0}$. We have
\begin{align}
&\Omega_{K}\bigg(n+\bigg(\frac{1+z}{2}\bigg)\bigg)=\sum_{i=0}^{\infty}{c_{n,i}z^{i}}\nonumber\\
&\qquad=1+K\int_{2}^{n+(1+z)/2}{\frac{\Omega_{K}(u-1)}{u-1}\mathrm{d}u}\nonumber\\
&\qquad=1+K\int_{2}^{n}{\frac{\Omega_{K}(u-1)}{u-1}\mathrm{d}u}+K\int_{n}^{n+(1+z)/2}{\frac{\Omega_{K}(u-1)}{u-1}\mathrm{d}u}\nonumber\\
&\qquad=\Omega_{K}(n)+K\int_{u=n}^{u=n+(1+z)/2}{\frac{\Omega_{K}(u-1)}{u-1}\mathrm{d}u}\nonumber\\
%&\qquad=\Omega_{K}(n)+K\int_{u-n=0}^{u-n=(1+z)/2}{\frac{\Omega_{K}(u-n-1+n)}{u-n-1+n}\mathrm{d}(u-n)}\nonumber\\
%&\qquad=\Omega_{K}(n)+K\int_{v=0}^{v=(1+z)/2}{\frac{\Omega_{K}(v+n-1)}{v+n-1}\mathrm{d}v}\quad\text{with $u=v+n$}\nonumber\\
&\qquad=\Omega_{K}(n)+K\int_{v=-1}^{v=z}{\frac{\Omega_{K}(n-1+(v+1)/2)}{2n-1+v}\mathrm{d}v}\quad\text{with $v=2u-2n-1$}\nonumber\\
&\qquad=\Omega_{K}(n)+\frac{K}{2n-1}\int_{u=-1}^{u=z}{\sum_{i=0}^{\infty}{c_{n-1,i}u^{i}}\sum_{i=0}^{\infty}{\frac{(-1)^{i}u^{i}}{(2n-1)^{i}}}\mathrm{d}u}\nonumber\\
&\qquad=\Omega_{K}(n)+\frac{K}{2n-1}\int_{u=-1}^{u=z}{\sum_{i=0}^{\infty}{\bigg(\sum_{j=0}^{i}{\frac{(-1)^{i-j}}{(2n-1)^{i-j}}c_{n-1,j}}\bigg)u^{i}\mathrm{d}u}}\nonumber\\
&\qquad=\Omega_{K}(n)+\frac{K}{2n-1}\int_{u=-1}^{u=z}{\sum_{i=0}^{\infty}{\alpha_{i}u^{i}}\mathrm{d}u}\nonumber\\
&\qquad=\Omega_{K}(n)-\frac{K}{2n-1}\sum_{i=0}^{\infty}{\frac{(-1)^{i+1}\alpha_{i}}{i+1}}+\frac{K}{2n-1}\sum_{i=0}^{\infty}{\frac{\alpha_{i}z^{i+1}}{i+1}}.\label{permanent_brown_trace}
\end{align}
The continuity $\Omega_{K}$ implies that
\begin{displaymath}
\Omega_{K}(n)=\lim_{z\to 1}{\Omega_{K}\bigg(n-1+\frac{1+z}{2}\bigg)}=\lim_{z\to 1}{\sum_{i=0}^{\infty}{c_{n-1,i}z^{i}}}=\sum_{i=0}^{\infty}{c_{n-1,i}}.
\end{displaymath}
Hence (\ref{permanent_brown_trace}) is rewritten as
\begin{align*}
\Omega_{K}\bigg(n+\frac{1+z}{2}\bigg)&=\sum_{i=0}^{\infty}{c_{n-1,i}}-\frac{K}{2n-1}\sum_{i=0}^{\infty}{\frac{\alpha_{i}(-1)^{i+1}}{i+1}}+\frac{K}{2n-1}\sum_{i=0}^{\infty}{\frac{\alpha_{i}z^{i+1}}{i+1}}\\
&=c_{n,0}+\sum_{i=1}^{\infty}{\frac{K\alpha_{i-1}}{(2n-1)i}z^{i}}=c_{n,0}+\sum_{i=1}^{\infty}{c_{n,i}z^{i}}.
\end{align*}
This concludes the proof.
\end{proof}

For instance, by reading $\Omega_{1}(2^{13})$ from the left half of Table \ref{tab_OmegaK} and recalling (\ref{ratio_limit}), the proportion of random permutations on at least $2^{14}$ elements, and with a cycle of smallest length at least $2^{13}$ is close to $1/\Omega_{1}(2^{13})\approx 0.000218 $. We note that if the number of permuted elements is exactly $2^{14}$, then there will be no smallest component of size at least $2^{13}$; one can observe this from the recurrence relation in Section \ref{subsec_rec_reln} as well.

Similarly by reading $\Omega_{1/2}(2^{13})$ from the right half of Table \ref{tab_OmegaK} and recalling (\ref{ratio_limit}), the proportion of random $2$-regular graphs with at least $2^{14}$ vertices, and with a large smallest component of at least $2^{13}$ is close to $1/\Omega_{1/2}(2^{13})\approx 0.0131$. We note that if the number of vertices is exactly $2^{14}$, then there will be no smallest component of size at least $2^{13}$.

\begin{center}
\begin{longtable}{l|l||l|l|c|l|l||l|l}
\multicolumn{9}{c}{Table \label{tab_OmegaK}\NewTabulare{}: A few values of $\Omega_{K}(x)$ for $K=1$ and $K=1/2$}\\
\hline
\multicolumn{4}{c|}{$K=1$} & & \multicolumn{4}{c}{$K=1/2$}\\
\hline
$x$ & $\Omega_{K}(x)$ & $x$ & $\Omega_{K}(x)$  & & $x$ & $\Omega_{K}(x)$ & $x$ & $\Omega_{K}(x)$\\
\hline
\endfirsthead
\multicolumn{4}{c|}{$K=1$} & & \multicolumn{4}{c}{$K=1/2$}\\
\hline
$x$ & $\Omega_{K}(x)$ & $x$ & $\Omega_{K}(x)$  & & $x$ & $\Omega_{K}(x)$ & $x$ & $\Omega_{K}(x)$\\
\hline
\endhead
\hline \multicolumn{9}{c}{Continued on next page}
\endfoot
\endlastfoot
$1$ & $1$ & $16$          & $8.9874$     & & $1$ & $1$ & $16$          & $3.3302$\\
$2$ & $1$ & $32$          & $17.9749$    & & $2$ & $1$ & $32$          & $4.7470$\\
$3$ & $1.6941$ & $64$     & $35.9498$    & & $3$ & $1.3470$ & $64$     & $6.7397$\\
$4$ & $2.2468$ & $128$    & $71.8997$    & & $4$ & $1.5866$ & $128$    & $9.5501$\\
$5$ & $2.8085$ & $256$    & $143.7995$   & & $5$ & $1.7971$ & $256$    & $13.5191$\\
$6$ & $3.3703$ & $512$    & $287.5991$   & & $6$ & $1.9856$ & $512$    & $19.1282$\\
$7$ & $3.9320$ & $1024$   & $575.1983$   & & $7$ & $2.1579$ & $1024$   & $27.0580$\\
$8$ & $4.4937$ & $2048$   & $1150.3966$  & & $8$ & $2.3175$ & $2048$   & $38.2705$\\
$9$ & $5.0554$ & $4096$   & $2300.7932$  & & $9$ & $2.4669$ & $4096$   & $54.1260$\\
$10$ & $5.6171$ & $8192$  & $4567.8834$  & & $10$ & $2.6077$ & $8192$  & $76.5480$\\
\end{longtable}
\end{center}

We conclude this section by mentioning that \cite{Finch_2021_1} gives values for $1/\Omega_{K}(x)$ with $x=2,3,4,5$, and that, if we invert values from Table \ref{tab_OmegaK} for $x=2,3,4,5$, they agree with those from \cite{Finch_2021_1}.

\section{Conclusion}

In this paper, we computed the normalization constant of the variance of the distribution of the smallest component of random combinatorial objects. We used different approaches: an analytic method based on the singularity analysis for generating functions, a numerical integration method using Taylor expansions for the Buchstab function, and by using the recurrence relation for counting the number of smallest components. All the methods yield to $1.3070\ldots$ We also showed how to compute the value of the generalized Buchstab function by building recursively sequences of Taylor expansions for each unit interval of the form $[n,n+1)$ where $n\in\mathbb{N}\setminus\{0\}$. By obtaining very accurate values of the generalized Buchstab function, we can compute the asymptotic proportion of large smallest components for various kinds of random combinatorial objects.

\section*{Acknowledgements}

D. Panario is partially funded by the Natural Science and Engineering Research Council of Canada, reference number RPGIN-2018-05328. The authors
thank an anonymous referee for several suggestions and corrections
that improved the paper.

%%%%%%%%%%%%%%%%%%%%%%%%%%%%%%%%%%%%%%%%%%%%%%%%%%%%%%%%%%%%%%%%%%%%%%%%%%%%%%%%
% GHOSTS REF
%\cite{flajolet2001complete}
%%%%%%%%%%%%%%%%%%%%%%%%%%%%%%%%%%%%%%%%%%%%%%%%%%%%%%%%%%%%%%%%%%%%%%%%%%%%%%%%

%\addcontentsline{toc}{chapter}{References}
\settocbibname{References}
\bibliographystyle{plain}
%\bibliography{ref_list}
%\input{buchstab_const_v6.bbl}
\newcommand{\SortNoop}[1]{}

%\appendix
%\chapter{app_A}

\end{document}